\documentclass[12pt]{article}
\usepackage{mathrsfs}
\usepackage{amssymb}
\usepackage{color}

\usepackage{amsthm}
\usepackage{amsmath}
\allowdisplaybreaks[4]
\usepackage{graphicx}

 \newtheorem{thm}{Theorem}[section]
 
 \newtheorem{lem}[thm]{Lemma}
 
 \theoremstyle{definition}
 
 \newtheorem{rem}[thm]{Remark}
 \numberwithin{equation}{section}

\def\u0{u_{0}}
\def\l2{l_{2}}

\def\t{\tilde}

\theoremstyle{definition}

\theoremstyle{remark}

\makeatletter

\newcommand{\Rmnum}[1]{\expandafter\@slowromancap\romannumeral #1@}
\makeatother

\begin{document}
\title{Improved Liouville theorems for axially symmetric Navier-Stokes equations\footnote{This is the English translation of the paper in Science China Mathematics, Vol. 47, no. 10, 2017, a special issue dedicated to Professor  Li Ta-Tsien, for his 80th birthday.}}
\author{Zhen Lei\footnote{School of Mathematical Sciences; LMNS and Shanghai Key Laboratory for Contemporary Applied Mathematics, Fudan University, Shanghai 200433, P. R.China. Email: zlei@fudan.edu.cn}, Qi S. Zhang\footnote{Department of Mathematics, University of California, Riverside, CA 92521, USA. Email: qizhang@math.ucr.edu}, Na Zhao\footnote{School of Mathematics, Fudan University, Shanghai 200433, P. R. China. Email: nzhao13@fudan.edu.cn.}}

\date{}

\maketitle

\begin{abstract}
  In this paper, we consider the Liouville property for ancient solutions of the incompressible Navier-Stokes equations. In 2D and the 3D axially symmetric case without swirl, we prove sharp Liouville theorems for smooth ancient mild solutions: velocity fields $v$ are constants if vorticity fields satisfy certain condition and $v$ are sublinear with respect to spatial variables, and we also give counterexamples when $v$ are linear with respect to spatial variables. The condition which vorticity fields need to satisfy is
  $\lim\limits_{|x|\rightarrow +\infty}|w(x,t)|=0$ and
  $\lim\limits_{r\rightarrow +\infty}\frac{|w|}{\sqrt{x_1^2+x_2^2}}=0$ uniformly for all $t\in(-\infty,0)$ in 2D and 3D axially symmetric case without swirl, respectively.
  In the case when solutions are axially symmetric with nontrivial swirl, we prove that if $\Gamma=rv_\theta\in L^\infty_tL^p_x(\mathbb{R}^3\times(-\infty,0))$ where $1\leq p<\infty$, then bounded ancient mild solutions are constants.

\end{abstract}

\maketitle





\section{Introduction}
The Navier-Stokes equations describe the time evolution of solutions of mathematical models of viscous incompressible fluids. In Cartesian coordinates, with viscosity constant being 1, they are
\begin{equation}\label{2DNS}
  \begin{cases}
    v_t+v\cdot \nabla v+\nabla p=\Delta v,\\
    \nabla\cdot v=0,
  \end{cases}
\end{equation}
where $v(x,t)$ represents the velocity field and $p(x,t)$ represents the pressure.
As is well known, the Navier-Stokes equations are globally well-posed in 2D. However, the global well-posedness is still unclear in 3D even if we consider axially symmetric solutions.

In this paper, we consider the Liouville property for ancient solutions of the incompressible Navier-Stokes equations. In 2D and the 3D axially symmetric swirl free case, we prove sharp Liouville theorems for smooth ancient mild solutions: velocity fields $v$ are constants if vorticity fields satisfy certain decay or growth condition and $v$ are sublinear with respect to spatial variables, and we also give counterexamples when $v$ are linear with respect to spatial variables. In the case when solutions are axially symmetric with nontrivial swirl, we prove that if $\Gamma=rv_\theta\in L^\infty_tL^p_x(\mathbb{R}^3\times(-\infty,0))$ where $1\leq p<\infty$, or $\lim\limits_{r\rightarrow\infty}\Gamma = 0$ uniformly, then bounded ancient mild solutions are constants.

In order to present the result precisely, let us first recall some basic concepts of axially symmetric solutions.
In cylindrical coordinates $r,\theta,z$ with $(x_1,x_2,x_3)=(r\cos\theta,r\sin\theta,z)$, axially symmetric solutions are of the form
\begin{equation}\nonumber
\begin{cases}
  v(x, t) = v_r(r, z, t)e_r +  v_\theta(r, z, t)e_\theta +  v_z(r,
z, t)e_z,\\
  p(x,t)=p(r,z,t).
\end{cases}
\end{equation}
The components $v_r,v_\theta,v_z,p$ are all independent of the angle of rotation $\theta$. Here $r = r(x) = \sqrt{x_1^2 + x_2^2}$ and
$\{e_r,e_\theta, e_z\}$ are the basis vectors for $\mathbb{R}^3$ given by
\begin{equation}\nonumber
e_r = \begin{pmatrix}\frac{x_1}{r}\\ \frac{x_2}{r}\\ 0
  \end{pmatrix},
\quad e_\theta = \begin{pmatrix}- \frac{x_2}{r}\\ \frac{x_1}{r}\\
0
  \end{pmatrix},
\quad e_z = \begin{pmatrix}0\\ 0\\ 1
  \end{pmatrix}.
\end{equation}
It is well known that $v_r,v_\theta,v_z$ satisfy
\begin{equation}\label{axi-NS}
\begin{cases}
\partial_tv_r + (v_re_r + v_ze_z)\cdot\nabla v_r - \frac{(v_\theta)^2}{r}
  + \partial_rp = \big(\Delta - \frac{1}{r^2}\big)v_r,\\[-4mm]\\
\partial_tv_\theta + (v_re_r + v_ze_z)\cdot\nabla v_\theta + \frac{v_rv_\theta}{r}
  = \big(\Delta - \frac{1}{r^2}\big)v_\theta,\\[-4mm]\\
\partial_tv_z + (v_re_r + v_ze_z)\cdot\nabla v_z + \partial_zp = \Delta v_z,\\[-4mm]\\
\partial_rv_r + \frac{v_r }{r} + \partial_zv_z = 0,
\end{cases}
\end{equation}
where $\Delta$ is the cylindrical scalar Laplacian and $\nabla$ is the cylindrical gradient field:
\begin{equation*}
  \Delta=\partial_r^2+\frac{1}{r}\partial_r+\frac{1}{r^2}\partial^2_\theta+\partial_z^2,
  \,\,\nabla=(\partial_r,\frac{1}{r}\partial_\theta,\partial_z).
\end{equation*}

Recall that if the swirl $v_\theta=0$, equations \eqref{axi-NS} then can be written as
\begin{equation}\label{noswirlequ}
\begin{cases}
\partial_tv_r + (v_re_r + v_ze_z)\cdot\nabla v_r
  + \partial_rp = \big(\Delta - \frac{1}{r^2}\big)v_r,\\[-4mm]\\
\partial_tv_z + (v_re_r + v_ze_z)\cdot\nabla v_z+\partial_zp = \Delta v_z,\\[-4mm]\\
\nabla\cdot v = \partial_rv_r +
\frac{v_r}{r} + \partial_zv_z = 0.
\end{cases}
\end{equation}
For \eqref{noswirlequ} where the swirl component of the velocity field is trivial, independently, Lady$\mathrm{\check{z}}$henskaya \cite{L1968} and Uchoviskii, Yudovich \cite{UY} proved that weak solutions are regular for all time (see also \cite{LMNP}). If the swirl $v_\theta\not\equiv 0$, the global well-posedness of \eqref{axi-NS} is still unclear. Recently, tremendous efforts and interesting progresses have been made on the regularity problem of the axially symmetric Navier-Stokes equations with nontrivial swirl, for example, \cite{BZ,CL,CSYT2008,CSYT2009,HL,HLL,KNSS,LZ,WZ,S,JX,LNZ,LZ-2,LMNP,NP,NP-2,SS,TX,ZZ}. We point out that in \cite{CSYT2008} and \cite{CSYT2009}, Chen, Strain, Tsai and Yau proved that suitable weak solutions of \eqref{axi-NS} are smooth if the velocity field $v$ satisfies $r|v| \leq C <\infty$.
Their method is based on the classical results by Nash \cite{Nash}, Moser \cite{Moser} and De Giorgi \cite{DG}. In \cite{KNSS}, Koch, Nadirashvili, Seregin and $\mathrm{\check{S}}$ver$\mathrm{\acute{a}}$k proved the same result using a totally different method. The core of their proof is the Liouville theorem for ancient solutions of the Navier-Stokes equations. Later, in \cite{LZ}, when the velocity field belongs to $L^\infty(BMO^{-1})$, Lei and Zhang obtained the corresponding  Liouville theorem and global well-posedness of solutions.  In other words, they solved the global regularity problem of $L^\infty([0,\infty);BMO^{-1})$ solutions of the axially symmetric Navier-Stokes equations based on their proof of the corresponding Liouville theorem. Chae also considered other Liouville type theorem, see \cite{CD}.

In order to describe the result in \cite{KNSS} precisely, we first recall some basic concepts used on the Liouville theorem of the Navier-Stokes equations. A velocity field is called an ancient solution if it exists in the time interval $(-\infty,0)$ and it satisfies the Navier-Stokes equations in some sense. An ancient weak solution is defined to be the ancient solution satisfying the equations \eqref{2DNS} in the sense of distributions. If there is a sequence $t_k\rightarrow -\infty$ such that $v(\cdot,t_k)$ is well defined and $v$ is a mild solution of the Cauchy problem in $\mathbb{R}^n\times(-\infty,0)$ with initial data $v(\cdot,t_k)$, i.e., $v$ satisfies
\begin{equation*}
  v(t)=S(t-t_k)v(t_k)+\int_{t_k}^tS(t-s)P\nabla\cdot(v(s)\otimes v(s))\mathrm{d}s,\ \ t\in(t_k,0),
\end{equation*}
where $S$ denotes the heat operator and $P$ the Helmholtz projection, then $v$ is an ancient mild solution of \eqref{2DNS}. We know from the above definition that ancient weak solutions may permit ``parasitic solutions'' which are of the form
\begin{equation*}
  v(x,t)=b(t),\,\,p(x,t)=-b'(t)x.
\end{equation*}
However, ancient mild solutions can eliminate such solutions.

In \cite{KNSS}, Koch, Nadirashvili, Seregin and
$\mathrm{\check{S}}$ver$\mathrm{\acute{a}}$k
stated a conjecture on the Liouville theorem for the axially symmetric Navier-Stokes equations: bounded, mild, ancient velocity fields are constants. Moreover, the authors of \cite{KNSS} proved such kind of Liouville theorems in 2D case and in the 3D axially symmetric swirl free case, or under the condition $r|v|\leqslant C$.
After that, the authors of \cite{LZ} proved the Liouville theorem for the axially symmetric Navier-Stokes equations when the stream function belongs to $BMO$ space (see also the new proofs of the result given by Seregin\cite{S} and Wang, Zhang\cite{WZ}).

The first main result in our paper is a generalization of the conjecture on the Liouville theorem in \cite{KNSS}. We prove, in 2D and the 3D axially symmetric swirl free case, that Liouville theorems hold for smooth ancient mild solutions of the incompressible Navier-Stokes equations if velocity fields are sublinear with respect to spatial variables and vorticity fields satisfy certain decay or growth condition (see Theorem \ref{2Dresult} and Theorem \ref{3Dnoswirl}). We remark that, unlike Liouville theorems in \cite{KNSS}, we don't need solutions being bounded. Moreover, we will give counterexamples when velocity fields are linear with respect to spatial variables. This shows that, the sublinear condition is sharp for Liouville theorems, which is also the case for harmonic functions.

The other main result in our paper is a Liouville theorem for bounded ancient solutions of the axially symmetric Navier-Stokes equations with nontrivial swirl (see Theorem \ref{3Dswirl}). Let $v$ be a bounded ancient mild solution of the axially symmetric Navier-Stokes equations with $v_\theta\neq 0$ and let $\Gamma=rv_\theta$. We prove that if $\Gamma\in L^\infty_tL^p_x$ where $1\leq p<\infty$, then $v$ must be a constant. As mentioned above, a conjecture on the Liouville theorem was proposed in \cite{KNSS}. Actually, in the 3D axially symmetric case, one can add the condition $\Gamma\in L^\infty_tL^\infty_x$ on the conjecture without loss much generality. The reason is $\Gamma$ is scaling invariant and it also satisfies the maximum principle. When $v_\theta\neq 0$, the Liouville theorem was proved in \cite{KNSS} under the condition $|v|\leq \frac{C}{r}$. We remark that, in our result, we only need the condition on one component $v_\theta$ of the velocity $v$ while no additional conditions are added on the other two components. Moreover, even though we haven't totally proved the conjecture in \cite{KNSS}, our result can still be considered as a progress on the conjecture since we prove the conjecture under the condition $\Gamma\in L^\infty_tL^p_x$ where $p$ can be any finite number. It seems that our result is apparently slightly weaker than the conjecture in \cite{KNSS}.

We are ready to state the main results in this paper.
\begin{thm}\label{2Dresult}
  Let $v$ be a smooth ancient mild solution of the 2D incompressible Navier-Stokes equations and let $w=\nabla\times v$ be the vorticity. If \begin{equation}\label{w}
   \lim\limits_{|x|\rightarrow +\infty}|w(x,t)|=0,
  \end{equation}
  uniformly for all $t\in(-\infty,0)$, then $w\equiv 0$ and $v$ is harmonic.

  If, in addition, $v$ satisfies
  \begin{equation}\label{sublinear}
  \lim\limits_{|x|\rightarrow +\infty}\frac{|v(x,t)|}{|x|}=0,
 \end{equation}
 uniformly for all $t\in(-\infty,0)$, then $v$ must be a constant.
\end{thm}

\begin{thm}\label{3Dnoswirl}
  Let $v$ be a smooth ancient mild solution of the 3D axially symmetric Navier-Stokes equations without swirl and let $w=\nabla\times v=w_{\theta}e_{\theta}$ be the vorticity. Define $\Omega=\frac{w_\theta}{r}$. If
  \begin{equation}\label{w2}
   \lim\limits_{r\rightarrow +\infty}|\Omega|=0,
  \end{equation}
  uniformly for all $t\in(-\infty,0)$, then $w_\theta\equiv 0$ and $v$ is harmonic.

  If, in addition, $v$ satisfies
  \begin{equation}\label{sublinear2}
   \lim\limits_{|x|\rightarrow +\infty}\frac{|v(x,t)|}{|x|}=0,
  \end{equation}
  uniformly for all $t\in(-\infty,0)$, then $v$ must be a constant.
\end{thm}
\begin{thm}\label{3Dswirl}
  Let $v$ be a bounded ancient mild solution of the 3D axially symmetric Navier-Stokes equations with $v_\theta\neq 0$ and let $\Gamma=rv_\theta$. If $\Gamma\in L^\infty_tL^p_x(\mathbb{R}^3\times(-\infty,0))$ where $1\leq p<\infty$, then $v$ must be a constant.
\end{thm}

\begin{rem}
  If $\lim\limits_{r\rightarrow\infty}\Gamma =0$ uniformly, the conclusion in Theorem \ref{3Dswirl} still holds.
\end{rem}
Let us introduce some notations used throughout the paper. Let $X_0=(x_0,t_0)$ be a space-time point in $\mathbb{R}^3\times\mathbb{R}$. For $R>0$, define $B_R(x_0)\subset \mathbb{R}^3$ as the ball of radius $R$ centered at $x_0$. The parabolic cylinder is defined as $Q_R(x_0,t_0)=B_R(x_0)\times(t_0-R^2,t_0)$ centered at $X_0$. If the center is the origin, we use the abbreviations $B_R=B_R(0)$ and $Q_R=Q_R(0,0)$.
$C$ represents an absolute constant above and in the sequel whose value may change from line to line.

The paper is organized as follows. In Section 2, we will make a summary about
current research on ancient solutions and relations between blowup of solutions and ancient solutions. In Section 3, we prove the Liouville theorem for smooth ancient mild solutions in 2D case, i.e., Theorem \ref{2Dresult}. In Section 4, we prove the Liouville theorem for smooth ancient mild solutions in the 3D axially symmetric swirl free case, i.e., Theorem \ref{3Dnoswirl}. In Section 5, we will consider axially symmetric solutions with nontrivial swirl and prove that bounded ancient mild solutions must be constants under the condition $\Gamma\in L^\infty_tL^p_x$, i.e., Theorem \ref{3Dswirl}.

\section{Preliminaries}

In this section, we first explain the relation between Liouville theorems for ancient solutions and the regularity or blowup problem of solutions, especially the relation with possible blow-up rate. Then, we describe results of Liouville theorems for bounded ancient mild solutions in \cite{KNSS}.

\subsection{Blow-up rate}

As well known, solutions of the Navier-Stokes equations are scaling invariant in the following sense. Let $(v,p)$ be a solution of \eqref{2DNS}, then for any $\lambda>0$, the following rescaled pair $(v_{\lambda},p_{\lambda})$ is also a solution:
\begin{equation}\label{scaling}
  v_{\lambda}(x,t)=\lambda v(\lambda x,\lambda^2t),\quad p_{\lambda}(x,t)=\lambda^2 p(\lambda x,\lambda^2t).
\end{equation}
Suppose a solution of the Navier-Stokes equations $v(x,t)$ blows up at $X_0=(x_0,t_0)$. Leray \cite{Leray} proved that there exists a constant $\epsilon_0>0$ such that the blow-up rate is at least
\begin{equation}\nonumber
  \|v(\cdot,t)\|_{L^{\infty}_x}\geqslant\frac{\epsilon_0}{\sqrt{t_0-t}}.
\end{equation}
Caffarelli, Kohn, Nirenberg \cite{CKN} proved that for such a blow-up solution, there exists a constant $\epsilon_1>0$, such that the average of $v$ over  $Q_R(x_0,t_0)$ satisfies
\begin{equation}\nonumber
  \left(\frac{1}{|Q_R|}\int_{Q_R(x_0,t_0)}|v|^3+|p|^\frac{3}{2}\mathrm{d}x\mathrm{d}t\right)^\frac{1}{3}\geqslant
  \frac{\epsilon_1}{R}.
\end{equation}
A natural and meaningful guess is that if a solution blows up at $(x_0,t_0)$, then the blow-up rate is at least
\begin{equation}\label{BUR}
  |v(x,t)|\thicksim\frac{O(1)}{\sqrt{(x_0-x)^2+(t_0-t)}}.
\end{equation}
Actually, this blow-up rate is invariant under the natural scaling
\eqref{scaling} of the Navier-Stokes equations. In other words, there holds the following formula for
$v_{\lambda}(x,t)$ obtained by \eqref{scaling}:
\begin{equation}\nonumber
  |v_{\lambda}(x,t)|=|\lambda v(\lambda x,\lambda^2 t)|\thicksim\frac{O(1)}{\sqrt{(x_0-x)^2+(t_0-t)}}.
\end{equation}

The Serrin-type criteria (see \cite{Serrin,FJR,Giga,GKT,L1967,Sohr,Struwe}) shows that a solution $v$ is regular if it satisfies the following estimate in the parabolic cylinder $Q_1$:
\begin{equation}\label{serrin}
  \|v\|_{L^s_tL^q_x(Q_1)}<\infty,\quad \frac{3}{q}+\frac{2}{s}\leqslant 1,\quad 2\leqslant s<\infty,\
  3<q\leqslant\infty,
\end{equation}
where we have defined
\begin{equation*}
  \|v\|_{L^s_tL^q_x(\Omega\times(t_1,t_2))}=\left\|\|v\|_{L^q_x(\Omega)}\right\|_{L^s_t(t_1,t_2)},
\end{equation*}
for any domain $\Omega\subset \mathbb{R}^3$.
For any $X_0=(x_0,t_0)\in Q_1$, \eqref{serrin} shows the following local smallness of $v$:
\begin{equation}\label{smallnorm}
  \lim_{R\rightarrow 0}\|v\|_{L^s_tL^q_x(Q_R(x_0,t_0))}=0.
\end{equation}
Thus, \eqref{serrin} is called to be $\epsilon$-regularity criteria since it implies the norm's local smallness. For the end point case ($q=3,\,s=\infty$), \eqref{smallnorm} can not be obtained from \eqref{serrin}. The authors of \cite{ISS} proved the end point regularity criteria.

However, these criteria can not exclude the blow-up possibility with the natural scaling rate \eqref{BUR}. In other words, if a solution satisfies the following estimate:
\begin{equation}\label{BUR2}
  |v(x,t)|\leqslant\frac{C}{\sqrt{(x_0-x)^2+(t_0-t)}},
\end{equation}
it is unclear whether the solution is regular or not.
In the axially symmetric case, \cite{CSYT2008} and \cite{CSYT2009} ruled out singular solutions satisfying \eqref{BUR2}. Actually, it was proved in \cite{CSYT2009} that solutions are regular if they satisfy
\begin{equation}\label{BUR3}
  |v| \leq \frac{C}{r}.
\end{equation}
It should be pointed out that \eqref{BUR3} is weaker than \eqref{BUR2}.
The key point of the proof in \cite{CSYT2008} and \cite{CSYT2009} is the H$\mathrm{\ddot{o}}$lder continuity of $\Gamma=rv_{\theta}$ at $r=0,\,t=0$.
They proved that
\begin{equation}\label{gamma}
  |\Gamma|\leqslant Cr^\alpha,
\end{equation}
for some positive constant $\alpha$.\footnote{Recently, Lei, Zhang \cite{LZnew} proved that if $\Gamma\leq C_*|\ln r|^{-2}$, then the solution is regular (see also an improved result of Wei \cite{Wei}).}
It is known that, in the case when $v_{\theta}\not\equiv 0$, the global regularity of solutions of \eqref{axi-NS} is unknown. However, it follows from the partial regularity theory of \cite{CKN} that the singular points of axially symmetric solutions can only appear on the axis of symmetry. For the axially symmetric Navier-Stokes equations, $\Gamma$ is a special quantity and it satisfies the following parabolic equation:
\begin{equation}
  \partial_t\Gamma+b\cdot\nabla\Gamma+\frac{2}{r}\partial_r\Gamma=\Delta \Gamma,
\end{equation}
where $b=v_re_r+v_ze_z$. We remark that $\Gamma$ enjoys the maximum principle.
\eqref{gamma} breaks the scaling invariance, thus the problem can be transformed from order one to $\epsilon$-regularity. To prove the H$\mathrm{\ddot{o}}$lder continuity of $\Gamma$, the authors of \cite{CSYT2008} applied the De Giorgi-Moser iteration method while Nash's method was used in \cite{CSYT2009}.

\subsection{Liouville Theorem and Blow-up Rate}\label{LiouvilleBU}

Liouville theorem is also an important method to study the blow-up rate of solutions. The authors of \cite{KNSS} proved, by using Liouville theorem, that if a weak solution $v$ of the axially symmetric Navier-Stokes equations in $\mathbb{R}^3\times(0,T)$ belongs to $L^{\infty}_{x,t}(\mathbb{R}^3\times(0,T'))$ for each $T'<T$ and satisfies
\begin{equation}\label{sverak_rate}
  |v(x,t)|\leqslant\frac{C}{\sqrt{x_1^2+x_2^2}},\,\,\,(x,t)\in \mathbb{R}^3\times(0,T),
\end{equation}
then $v$ won't blow up at $T$.

Since \cite{KNSS} is closely related to the current paper, here for the completeness, we recall the main idea of the proof in \cite{KNSS}. The proof of the statement mainly depends on the rescaling procedure. For convenience, we assume that $v$ is a mild solution which is bounded in $\mathbb{R}^3\times(0,T')$. If it blows up at $T$, we can apply the rescaling procedure to construct a bounded ancient mild solution. Specifically speaking, it can be done as follows. Let $h(t)=\sup\limits_{x\in \mathbb{R}^3}|v(x,t)|$, it follows from Leray \cite{Leray} that
\begin{equation}
  h(t)\geqslant\frac{C}{\sqrt{T-t}},
\end{equation}
for some small positive constant $C$. Let $H(t)=\sup\limits_{0\leqslant s\leqslant t}h(s)$, then there exists a sequence $t_k\nearrow T$ such that $h(t_k)=H(t_k)$. Choose a sequence of numbers $\gamma_k\searrow 1$. For each $k$, set $N_k=H(t_k)$ and choose $x_k\in \mathbb{R}^3$ such that $$M_k=|v(x_k,t_k)|\geqslant \frac{N_k}{\gamma_k}.$$
Define
\begin{equation}
 \label{rescaling}
 v_k(y,s)=\frac{1} {M_k} v(x_k+\frac {y}{M_k},t_k + \frac{s}{M_k^2}),
\end{equation}
and it is easy to see that functions $v_k$ are defined in $\mathbb{R}^3\times(A_k,B_k)$, where
$$A_k=-M_k^2 t_k,\,\,B_k=M_k^2(T-t_k)\geqslant \frac{C^2}{\gamma_k^2}.$$
Moreover, we know from the definition of $v_k$ that they satisfy
\begin{equation}\label{vkbound}
 \mbox{$|v_k|\leqslant \gamma_k$ in $\mathbb{R}^3\times(A_k,0)$ and $|v_k(0,0)|=1$}.
\end{equation}
Meanwhile, $v_k$ are mild solutions of the Navier-Stokes equations in $\mathbb{R}^3\times(A_k,0)$ with initial data
$$v_{k0}(y)=\frac1 {M_k}v_0(x_k+\frac {y}{M_k}).$$
It can be proved that there is a subsequence of $v_k$ converging to an ancient mild solution $\bar{v}$ of the Navier-Stokes equations and
$\bar{v}$ satisfies
\begin{equation}\label{vbar}
  |\bar{v}(y,s)|\leqslant 1,\,\,(y,s)\in \mathbb{R}^3\times(-\infty,0)\,\,\text{and}\,|\bar{v}(0,0)|=1.
\end{equation}
Let us denote $x_k=(x'_k,x_{k3})$ with $x'_k=(x_{k1},x_{k2})$. It follows from \eqref{sverak_rate} that $|x'_k|\leqslant \frac C {M_k}$. This shows that functions $v_k$ are axially symmetric. Therefore the limit function $\bar{v}$ is also axially symmetric with respect to a suitable axis.
Also, $\bar{v}$ satisfies the scale-invariant bound \eqref{sverak_rate}.
To sum up, $\bar{v}$ is a bounded ancient mild solution of the axially symmetric Navier-Stokes equations satisfying \eqref{sverak_rate}.  Hence Liouville theorem implies $\bar{v}=0$. However, \eqref{vbar} shows $\bar{v}(0,0)=1$, which makes a contradiction. Therefore, the solution won't blow up at $T$.

\subsection{Liouville Theorem for Bounded Solutions}\label{LTBS}

In this subsection, we recall the Liouville theorems proved in \cite{KNSS}. For completeness, we will present both the conclusion and the idea of the proof in \cite{KNSS}. The authors of \cite{KNSS} stated that any ancient mild solution of the Navier-Stokes equations with bounded velocity is a constant. They proved the statement in 2D and the 3D axially symmetric case with additional conditions. Now let us recall Liouville theorems in different cases.

Case 1. 2D case.

In 2D, the Liouville theorem in \cite{KNSS} was stated as follows.
Let $v$ be a bounded ancient mild solution of the Navier-Stokes equations in 2D. Then $v$ is a constant.

The proof is based on the equation of the vorticity $w=\nabla\times v$ which is a scalar in 2D defined by
$$w=\partial_1u_2-\partial_2u_1,$$
satisfying
\begin{equation}\label{curlequ}
  w_t+v\cdot\nabla w-\Delta w=0.
\end{equation}
We remark that $w$ satisfies the Harnack inequality which can be used to show that if $w\not\equiv 0$, then $w$ has to be almost equal to its maximum or minimum in large space-time areas. This contradicts the boundedness of $v$. Therefore $w\equiv 0$. It then follows from
$$\nabla\cdot v=0,\,\,\nabla\times v=0$$
that
$$\Delta v=\nabla(\nabla\cdot v)-\nabla\times\nabla\times v=0$$
holds in the whole space. Then the boundedness of $v$ implies that $v$ is a constant.

Case 2. Axially symmetric swirl free case.

In the case when $v_\theta=0$, \cite{KNSS} proved the following Liouville theorem. Let $v$ be a bounded ancient weak solution of the axially symmetric Navier-Stokes equations without swirl. Then $v(x,t)=(0,0,b_3(t))$ for some bounded measurable function $b_3:(-\infty,0)\rightarrow \mathbb{R}$.

The proof mainly depends on the equation of $\Omega=\frac{w_\theta}{r}$.  We know that $\Omega$ satisfies
\begin{equation}\label{wtheta1}
  \partial_t\Omega+u_r\partial_r\Omega
  +u_z\partial_z\Omega=\Delta \Omega+\frac{2}{r}\partial_r\Omega.
\end{equation}
This equation can be viewed as an equation defined in 5D. Actually, let $z=y_5$ and $r=\sqrt{y_1^2+y_2^2+y_3^2+y_4^2}$, and we note that for any function $\widetilde{f}(y_1,y_2,y_3,y_4,y_5)$
$=f(r,z)$ in 5D, there holds
$$\Delta_y\widetilde{f}(y_1,y_2,y_3,y_4,y_5)=(\partial^2_r f+\frac{3}{r}\partial_rf+\partial^2_zf)(r,z).$$
Therefore, $\Delta+\frac{2}{r}\partial_r$ can be interpreted as the 5D Laplacian operator $\Delta_5$ and the equation \eqref{wtheta1} can be written as
\begin{equation}\label{wtheta2}
  \partial_t\Omega+u_r\partial_r\Omega
  +u_z\partial_z\Omega=\Delta_5\Omega.
\end{equation}
Similar to the 2D case, we can get $\Omega\equiv0$ Due to \eqref{wtheta2}. Thus $w_\theta=0$ and $w=0$. Since $v$ is also divergence free, we have $v=(0,0,b_3(t))$.

Case 3. Axially symmetric case with nontrivial swirl.

In the axially symmetric case with nontrivial swirl, the Liouville theorem was obtained in \cite{KNSS} by adding an additional condition. The theorem was stated as follows. Let $v$ be a bounded ancient weak solution of the axially symmetric Navier-Stokes equations satisfying
\begin{equation}\label{vcond}
  |v(x,t)|\leq \frac{C}{\sqrt{x_1^2+x_2^2}},\quad (x,t)\in\mathbb{R}^3\times(-\infty,0).
\end{equation}
Then $v=0$ in $\mathbb{R}^3\times(-\infty,0)$.

The key component of the proof is the equation of $\Gamma=rv_\theta$. It is known that $\Gamma$ satisfies
\begin{equation}\label{gammaequ}
  \partial_t\Gamma+v_r\partial_r\Gamma+v_z\partial_z\Gamma+\frac{2}{r}\partial_r\Gamma
  =\Delta\Gamma.
\end{equation}
From condition \eqref{vcond}, we know that
$$|\Gamma|\leq C.$$
For any $\lambda>0$, we define
$$v_{\lambda}(x,t)=\lambda v(\lambda x,\lambda^2t),\quad
\Gamma_{\lambda}(x,t)=\Gamma(\lambda x,\lambda^2t).$$
Then $\Gamma_\lambda$ and $u_\lambda$ also satisfy the equation \eqref{gammaequ}.
Moreover, in $\mathbb{R}^3\times(-\infty,0)$, there holds
\begin{equation*}
  |v_{\lambda}|\leq\frac{C}{r},\quad|\Gamma_{\lambda}|\leq C,
\end{equation*}
uniformly for $\lambda>0$.
Set $M=\sup\limits_{\mathbb{R}^3\times(-\infty,0)}\Gamma$,
and we can use the rescaling procedure to move the points where $\Gamma_\lambda$ approaches $M$ to the place near the $x_3$-axis.
Consequently $\Gamma_\lambda$ is very large near the $x_3$-axis which is unreasonable since $\Gamma_\lambda=0$ when $r=0$.
Therefore, $\Gamma$ must be 0, i.e., $v_\theta=0$. This shows that any solution of the equations has no swirl. It then follows from the result for $v_\theta=0$ and condition \eqref{vcond} that $v=0$.

\section{Liouville Theorem for the Navier-Stokes Equations in 2D}

In this section, we will prove the Liouville theorem for the Navier-Stokes equations in 2D, i.e., Theorem \ref{2Dresult}.

\begin{proof}[Proof of Theorem \ref{2Dresult}]

First, we use a contradiction argument to prove $w\equiv 0$.

If $w\not\equiv 0$, we can suppose that $$\sup\limits_{(x,t)\in \mathbb{R}^2\times(-\infty,0)}w(x,t)=M>0.$$
In 2D, $w$ satisfies the following equation:
\[
\partial_t w+v\cdot\nabla w-\Delta w=0.
\]
 Based on the maximum principle of parabolic equations, we know that $w$ can not reach its maximum at an interior point. Hence there exists a sequence of points $(x_k,t_k), |t_k|\rightarrow \infty, |x_k|\leq C$ such that
\[
\lim_{k\rightarrow +\infty}w(x_k,t_k)=M.
\]
Here $C>0$ is a fixed constant and $|x_k|\leq C$ is due to our assumption \eqref{w}.

Define $\tilde{w}_k$ and $\tilde{v}_k$ as follows
$$\tilde{w}_k(y,s)=w(x_k+y,t_k+s),$$
$$\tilde{v}_k(y,s)=v(x_k+y,t_k+s),$$
then there exist two subsequences which are still denoted by $\tilde{w}_k$ and $\tilde{v}_k$ such that the two subsequences converge to $\tilde{w}$ and $\tilde{v}$ in $C^{2,1,\alpha}_{local}$, respectively. Moreover, $\tilde{w}$ satisfies:
\[
\begin{cases}
&\partial_s \tilde{w}+\tilde{v}\cdot\nabla \tilde{w}-\Delta \tilde{w}=0,\quad (y,s) \in\mathbb{R}^2\times(-\infty,0), \\
&\t{w}(0,0)=M.
\end{cases}
\]
The above shows that $\t{w}$ reaches its maximum in interior, thus we have $\t{w}\equiv M$ by the maximum principle.

Hence, for any $\epsilon >0$ and fixed $r_0>0$, there exists a number $k_0$, such that if $k>k_0$, $$|\t{w}_k-M|<\epsilon,$$
holds on the parabolic cylinder $Q_{r_0}$.
This implies that $w$ satisfies
$$|w-M|<\epsilon,$$
on $Q_{r_0}(x_k,t_k)$.
This is a contradiction with our assumption that $\lim\limits_{|x|\rightarrow +\infty}|w(t,x)|=0$ uniformly for all $t\in (-\infty,0)$ when $r_0$ is sufficiently large and $|x_k|$ is bounded. Thus $w\equiv0$.

Because $$\nabla\cdot v=0,\,\,\nabla\times v=0,$$
we have that $\Delta v=0$ holds in the whole space which shows that $v$ is harmonic. Moreover, since $v$ satisfies \eqref{sublinear}, $v$ must be a constant. We complete the proof of the theorem.

\end{proof}

\begin{rem}
The condition $v$ being sublinear (condition \eqref{sublinear}) in Theorem \ref{2DNS} is very important and can not be removed, even when $w\equiv0$. Moreover, we can give an counterexample to show that if $v$ is linear with respect to spatial variables, then there exists a nontrivial ancient solution of the Navier-Stokes equations in 2D. Hence, the sublinear condition is sharp for our Liouville theorem in 2D.

For example, let
$$v=(x_1,-x_2),$$
$$p=-\frac{1}{2}x_1^2-\frac{1}{2}x_2^2,$$
then $w=\partial_1u_2-\partial_2u_1=0$, and $(v,p)$ satisfies
\[
\begin{cases}
&v\cdot\nabla v-\Delta v+\nabla p=0,\\
&\nabla\cdot v=0,
\end{cases}
\]
which shows that $v$ is a stationary solution of the 2D Navier-Stokes equations. However, $v$ is not a constant solution.

\end{rem}

\section{Liouville Theorem for Axially Symmetric Navier-Stokes Equations with $v_\theta= 0$}

In this section, we will prove the Liouville theorem for the axially symmetric Navier-Stokes equations without swirl, i.e., Theorem \ref{3Dnoswirl}. In this case, $v=v_re_r+v_ze_z$.

\begin{proof}[Proof of Theorem \ref{3Dnoswirl}]

First, we prove $w_\theta\equiv0$ by contradiction.
In this case, $\Omega$ satisfies the following equation:
$$\partial_t\Omega+u_r\partial_r\Omega
+u_z\partial_z\Omega=\Delta \Omega+\frac{2}{r}\partial_r\Omega.$$
As described in Section \ref{LTBS}, $\Delta+\frac{2}{r}\partial_r$ can be considered as the 5D Laplacian operator $\Delta_5$. Hence, with a slight abuse of notation, the above equation can be written as the following equation defined in $\mathbb{R}^5\times(-\infty,0)$:
\begin{equation*}
  \partial_t\Omega+u_r\partial_r\Omega
  +u_z\partial_z\Omega=\Delta_5\Omega.
\end{equation*}
If $w_{\theta}\not\equiv0$, without loss of generality, we assume that
$$\sup\limits_{(x,t)\in \mathbb{R}^5\times(-\infty,0)}\Omega=M_0>0.$$
 Owing to the maximum principle of parabolic equations, $\Omega$ can not get its maximum in interior. Thus there exists a sequence $x_k=(x'_k,z_k,0,0)\in \mathbb{R}^5$ $(x'_k=(x_{k1},x_{k2}))$ and $t_k$, such that $|x_k'|\leq C$, where $C$ is a uniform constant independent of $t_k,z_k$ and
$$\lim\limits_{k\rightarrow \infty}\Omega(x_k,t_k)=M_0.$$

Define
$$\widetilde{\Omega}_{k}(y,s)=\Omega(x_k+y,t_k+s),$$
$$\widetilde{v}_{k}(y,s)=v(x_k+y,t_k+s),$$
where $y=(y_1,y_2,y_3,0,0)\in \mathbb{R}^5$, then up to subsequences, $\widetilde{\Omega}_{k}$ and $\widetilde{v}_k$ converge to $\widetilde{\Omega}$ and $\widetilde{v}$ in $C^{2,1,\alpha}_{local}$, respectively, and $\widetilde{\Omega}$ satisfies
\begin{equation*}
  \begin{cases}
    \partial_s\widetilde{\Omega}+\widetilde{u}_r\partial_r\widetilde{\Omega}
  +\widetilde{u}_z\partial_z\widetilde{\Omega}=\Delta_5\widetilde{\Omega},
  \quad(y,s)\in \mathbb{R}^5\times(-\infty,0),\\
  \widetilde{\Omega}(0,0)=M_0.
  \end{cases}
\end{equation*}
Therefore, $\widetilde{\Omega}$ get its maximum in interior which implies $\widetilde{\Omega}\equiv M_0$ according to the maximum principle.

Let $B_{r}(y)\subset\mathbb{R}^5$ be a ball of radius $r$ centered at $y=(y_1,y_2,y_3,0,0)$, $Q_r(y,s)=B_r(y)\times(s-r^2,s)$.
It then follows from
$$\lim\limits_{k\rightarrow\infty}\widetilde{\Omega}_k=M_0>0,$$
that for any $\epsilon>0$ and fixed $r_0>0$, if $k$ is sufficiently large,
\[
|\widetilde{\Omega}_{k}-M_0|\leq\epsilon,
\]
holds on $Q_{r_0}(0,0)$.
This implies that $\Omega$ satisfies
$$|\Omega-M_0|<\epsilon,$$ on $Q_{r_0}(x_k,t_k)$.
When $r_0$ is sufficiently large, this is a contradiction with the assumption that $\Omega$ satisfies $\lim\limits_{r\rightarrow\infty}|\Omega|=0$ uniformly for $t$. Hence $w_\theta\equiv0$.

Because $v=v_re_r + v_ze_z$, we get $\Delta v=-\nabla\times(w_\theta e_\theta)=0$ which shows that $v$ is harmonic.
Moreover, it follows from the fact $v$ is sublinear with respect to $x$ and $\Delta v=0$ that $v$ must be a constant. We complete the proof of Theorem \ref{3Dnoswirl}.

\end{proof}

\begin{rem}
The first conclusion in Theorem \ref{3Dnoswirl} shows that if
$w_\theta$ is sublinear with respect to $r$, then $w_\theta\equiv0$.
We can give a counterexample to show that the conclusion will be wrong if $w_\theta$ is linear with respect to $r$, and consequently we can not get $v$ is harmonic. The above analysis tells us that condition \eqref{w2} is necessary. For example, let
$$v=(-x_1x_3,-x_2x_3,x_3^2),$$
$$p=-\frac 1 2 x_3^4+2x_3,$$
then $v_\theta=v\cdot e_\theta=0$ and $(v,p)$ satisfies
\[
\begin{cases}
&v\cdot\nabla v-\Delta v+\nabla p=0,\\
&\nabla\cdot v=0.
\end{cases}
\]
However $w_\theta=-r\not\equiv0$, $\Delta v=(0,0,2)\neq0$.

On the other hand, the condition $v$ being sublinear with respect to $x$ (condition \eqref{sublinear2}) is also necessary and can not be removed, even when $w_\theta=0$. The following is a counterexample to show that if $v$ is linear with respect to spatial variables, then there exists a nontrivial ancient solution of the axially symmetric Navier-Stokes equations without swirl. It then follows that the sublinear condition is sharp for our Liouville theorem. For example, let
$$v=(-\frac{1}{2}x_1,-\frac{1}{2}x_2,x_3),$$
$$p=\frac{1}{8}x_1^2+\frac{1}{8}x_2^2+\frac{1}{2}x_3^2,$$
then we have
$$v_{\theta}=v\cdot e_{\theta}=0,$$
$$v_{r}=v\cdot e_{r}=-\frac{1}{2}r,$$
$$v_z=x_3,$$
which implies $w_{\theta}=0$ and $(v,p)$ satisfies
\[
\begin{cases}
&v\cdot\nabla v-\Delta v+\nabla p=0,\\
&\nabla\cdot v=0.
\end{cases}
\]
Therefore $v$ is a stationary solution of the axially symmetric Navier-Stokes equations without swirl. However, $v$ is not a constant solution.

\end{rem}

\section{Liouville Theorem for Axially Symmetric Navier-Stokes Equations with $v_\theta\neq 0$}

In this section, we will prove the Liouville theorem for the axially symmetric Navier-Stokes equations with $v_\theta\neq 0$, i.e., Theorem \ref{3Dswirl}.
We split the proof of Theorem \ref{3Dswirl} into two parts (Lemma \ref{lem5.1} and Lemma \ref{lem5.2}) since the strategies used to prove these two lemmas are very different. The proof of Lemma \ref{lem5.1} is based on the Nash-Moser iteration method. While proving Lemma \ref{lem5.2}, we mainly use the maximum principle of parabolic equations and a contradiction argument.

\begin{lem}\label{lem5.1}
 Let $v$ be a bounded ancient mild solution of the axially symmetric Navier-Stokes equations. $\Gamma=rv_\theta$. If $\Gamma\in L^\infty_tL^p_x(\mathbb{R}^3\times (-\infty,0))$, where $1\leq p<\infty$,
then
\[
\lim\limits_{r\rightarrow\infty}\Gamma(x,t)=0
\]
holds uniformly for $t$ and $z$.
\end{lem}
\begin{lem}\label{lem5.2}
  Under the conditions of Lemma \ref{lem5.1}, $v$ must be a constant.
\end{lem}
Now we prove the above two lemmas.

\begin{proof}[Proof of Lemma\ref{lem5.1}]
  When $r$ is sufficiently large, we will deduce by two steps that
  \begin{equation}\label{target}
  |\Gamma(x,t)|\leq Cr^{-\frac{1}{p}}.
  \end{equation}
  This immediately proves our lemma.

  Step 1. Prove the mean value inequality \eqref{meanvalue}.

  We know that $\Gamma$ satisfies the following equation
  \begin{equation}\label{Gammaequ}
  \partial_t\Gamma+b\cdot\nabla\Gamma+\frac{2}{r}\partial_r\Gamma-\Delta\Gamma=0,
  \end{equation}
  where $b=v_re_r+v_ze_z$. We will derive an energy estimate of De Giorgi type for \eqref{Gammaequ} at a fixed point $X=(x_0,t_0)$ where $x_0=(x_{01},x_{02},x_{03})$. For this purpose, we need a refined cut-off function near $X$. Let $\frac{1}{2}\leq \sigma_2<\sigma_1\leq 1$. Define $\phi\in C^\infty_0(\mathbb{R}^3)$ and $\eta\in C^\infty_0(-\infty,0)$ satisfying
  \[
   \begin{cases}
  &\textmd{supp}\phi\subset B_{\sigma_1}(x_0),\phi=1~\text{in}~B_{\sigma_2}(x_0),0\leq\phi\leq1,\\
  &|\nabla\phi|\leq \frac{C}{\sigma_1-\sigma_2},~|\nabla^2\phi|\leq \frac{C}{(\sigma_1-\sigma_2)^2};\\
  &\textmd{supp}\eta\subset (t_0-\sigma^2_1,t_0),~\phi=1 ~\text{on} ~(t_0-\sigma^2_2,t_0),0\leq\eta\leq1,\\
  &|\eta '|\leq\frac{C}{(\sigma_1-\sigma_2)^2}.\\
  \end{cases}
  \]
  Let $\psi(t,x)=\eta(t)\phi(x)$ be the cut-off function. Define $Q_{\sigma_i}(x_0,t_0)= B_{\sigma_i}(x_0)\times(t_0-\sigma^2_i,t_0)$, $i=1,2$.
  Testing \eqref{Gammaequ} by $2p\Gamma^{2p-1}\psi^2$ and integrating it on $\mathbb{R}^3\times(-\infty,0)$ gives
  \[
   \iint(\partial_t\Gamma+b\cdot\nabla\Gamma+\frac{2}{r}\partial_r\Gamma-\Delta\Gamma)2p\Gamma^{2p-1}\psi^2 \mathrm{d}x\mathrm{d}t=0.
  \]
  By calculating each term on the above, we get
  \begin{equation}\label{term1}
   \iint\partial_t\Gamma \cdot 2p\Gamma^{2p-1}\psi^2 \mathrm{d}x\mathrm{d}t=\int\Gamma^{2p}\psi^2|^{t_0}_{t_0-\sigma^2_1} \mathrm{d}x-\iint\Gamma^{2p}\partial_t\psi^2\mathrm{d}x\mathrm{d}t,
  \end{equation}
  \begin{equation}\label{term2}
   \iint\ (b\cdot\nabla)\Gamma\cdot 2p\Gamma^{2p-1}\psi^2 \mathrm{d}x\mathrm{d}t
    =\iint\ (b\cdot\nabla)\Gamma^{2p}\psi^2\mathrm{d}x\mathrm{d}t
    =-2\iint\ \Gamma^{2p}\psi(b\cdot\nabla)\psi \mathrm{d}x\mathrm{d}t,
  \end{equation}
  \begin{equation}\label{term3}
   \begin{split}
    \iint\ \frac{2}{r}\partial_r\Gamma\cdot 2p\Gamma^{2p-1}\psi^2 \mathrm{d}x\mathrm{d}t
   =4\pi\iiint \partial_r\Gamma^{2p}\psi^2 \mathrm{d}r\mathrm{d}z\mathrm{d}t
   =-4\iint\ \Gamma^{2p}\frac{\psi\partial_r\psi}{r} \mathrm{d}x\mathrm{d}t,
  \end{split}
 \end{equation}
  \begin{equation}\label{term4}
  \begin{split}
   \iint\ -\Delta\Gamma\cdot 2p\Gamma^{2p-1}\psi^2 \mathrm{d}x\mathrm{d}t
  =&\frac{2(2p-1)}{p}\iint\ |\nabla(\Gamma^p\psi)|^2 \mathrm{d}x\mathrm{d}t
  -2\iint \Gamma^{2p}|\nabla \psi|^2\mathrm{d}x\mathrm{d}t\\
  &+\frac{2(p-1)}{p}\iint\ \Gamma^{2p} \psi\Delta \psi \mathrm{d}x\mathrm{d}t.
  \end{split}
 \end{equation}
 It then follows from \eqref{term1}, \eqref{term2}, \eqref{term3} and \eqref{term4} that
 \begin{equation}\label{deGiorgi}
  \begin{split}
  &\sup\limits_{t\in (t_0-\sigma^2_1,t_0)}\parallel \Gamma^{p}\psi\parallel^2_{L^2_x(\mathbb{R}^3)}+\int_{Q_{\sigma_1}(x_0,t_0)} |\nabla(\Gamma^p\psi)|^2\mathrm{d}x\mathrm{d}t\\
  \leq &C\int_{Q_{\sigma_1}(x_0,t_0)}\Gamma^{2p}(\psi\Delta\psi+|\nabla\psi|^2+\psi \partial_t \psi+\frac{\psi \partial_r \psi}{r}+\psi(b\cdot\nabla)\psi) \mathrm{d}x\mathrm{d}t\\
  \leq&C\int_{Q_{\sigma_1}(x_0,t_0)}\Gamma^{2p}(\frac{1}{(\sigma_1-\sigma_2)^2}+\frac{1}{\sigma_1-\sigma_2}) \mathrm{d}x\mathrm{d}t\\
  \leq&\frac{C}{(\sigma_1-\sigma_2)^2}\int_{Q_{\sigma_1}(x_0,t_0)}\Gamma^{2p} \mathrm{d}x\mathrm{d}t,
  \end{split}
\end{equation}
where $C$ also depends on $\|b\|_{L^\infty_{t,x}}$.

We are now ready to derive the mean value inequality based on \eqref{deGiorgi} using Moser's iteration method.
By H$\mathrm{\ddot{o}}$lder's inequality and Sobolev imbedding theorem, one has
\begin{equation}\nonumber
 \begin{split}
 &\int_{Q_{\sigma_1}(x_0,t_0)}(\psi|\Gamma|^p)^3\mathrm{d}x\mathrm{d}t\\
 \leq & C\int^{t_0}_{t_0-\sigma^2_1}\|\psi\Gamma^p\|^{\frac{3}{2}}
 _{L^2_x(\mathbb{R}^3)}\|\psi\Gamma^p\|^{\frac{3}{2}}_{L^6_x(\mathbb{R}^3)}\mathrm{d}t\\
 \leq & C[\frac{1}{(\sigma_1-\sigma_2)^2}\int_{Q_{\sigma_1}(x_0,t_0)}\Gamma^{2p}\mathrm{d}x\mathrm{d}t] ^{\frac{3}{4}}\int^{t_0}_{t_0-\sigma^2_1}\|\nabla(\psi\Gamma^p)\|^{\frac{3}{2}}
 _{L^2_x(\mathbb{R}^3)}\mathrm{d}t\\
 \leq&C[\frac{1}{(\sigma_1-\sigma_2)^2}\int_{Q_{\sigma_1}(x_0,t_0)}\Gamma^{2p}\mathrm{d}x\mathrm{d}t]
 ^{\frac{3}{4}}(\int^{t_0}_{t_0-\sigma^2_1}\|\nabla(\psi\Gamma^p)\|^2_{L^2_x(\mathbb{R}^3)}\mathrm{d}t)
 ^{\frac{3}{4}}(\int^{t_0}_{t_0-\sigma^2_1}dt)^{\frac{1}{4}}\\
 \leq&C[\frac{1}{(\sigma_1-\sigma_2)^2}\int_{Q_{\sigma_1}(x_0,t_0)}\Gamma^{2p}\mathrm{d}x\mathrm{d}t]
 ^{\frac{3}{2}}.
 \end{split}
\end{equation}
Hence,
\begin{equation}\label{iteIneq}
  \int_{Q_{\sigma_2}(x_0,t_0)}|\Gamma|^{3p} \mathrm{d}x\mathrm{d}t\leq C[\frac{1}{(\sigma_1-\sigma_2)^2}\int_{Q_{\sigma_1}(x_0,t_0)}\Gamma^{2p}\mathrm{d}x\mathrm{d}t]
  ^{\frac{3}{2}}.
\end{equation}

Let $\sigma=\frac{1}{2}$ be a constant. For integers $j\geq0$, set
$\sigma_1=\frac{1}{2}(1+\sigma^j)$, $\sigma_2=\frac{1}{2}(1+\sigma^{j+1})$ and
$p=(\frac{3}{2})^j$ in \eqref{iteIneq}, we then get
\begin{equation}\nonumber
 \begin{split}
 (\int_{Q_{\frac{1}{2}(1+\sigma^{j+1})}(x_0,t_0)}|\Gamma|^{2(\frac{3}{2})^{j+1}}
 \mathrm{d}x\mathrm{d}t)
 ^{\frac{1}{2}(\frac{2}{3})^{j+1}}
 \leq&[C\frac{4}{\sigma^{2(j+1)}}\int_{Q_{\frac{1}{2}(1+\sigma^{j})}(x_0,t_0)}\Gamma
 ^{2(\frac{3}{2})^j}\mathrm{d}x\mathrm{d}t]^{\frac{1}{2}(\frac{2}{3})^j} \\
 \leq&4C^{\frac{1}{2}(\frac{2}{3})^j}2^{(\frac{2}{3})^j(j+1)}
 (\int_{Q_{\frac{1}{2}(1+\sigma^{j})}(x_0,t_0)}\Gamma^{2(\frac{3}{2})^j}\mathrm{d}x\mathrm{d}t)
 ^{\frac{1}{2}(\frac{2}{3})^j} \\
 \leq&4C^{\frac{1}{2}\sum\limits^j\limits_{k=0}(\frac{2}{3})^k}
 2^{\sum\limits^j\limits_{k=0}(\frac{2}{3})^k(k+1)}
 (\int_{Q_1(x_0,t_0)} \Gamma^2 \mathrm{d}x\mathrm{d}t)^{\frac{1}{2}} \\
 \leq&C(\int_{Q_1(x_0,t_0)} \Gamma^2 \mathrm{d}x\mathrm{d}t)^{\frac{1}{2}},
 \end{split}
\end{equation}
where $Q_1(x_0,t_0)=B_1(x_0)\times(t_0-1,t_0)$.
Let $j\rightarrow\infty$ in the above inequality, we obtain
\begin{eqnarray}
 \sup\limits_{Q_{\frac{1}{2}}(x_0,t_0)}|\Gamma| &\leq&C(\int_{Q_1(x_0,t_0)} \Gamma^2 \mathrm{d}x\mathrm{d}t)^{\frac{1}{2}}.\nonumber
\end{eqnarray}
From this, a well-known algebraic trick gives the following mean value inequality:
\begin{equation}\label{meanvalue}
 \sup\limits_{Q_{\frac{1}{2}}(x_0,t_0)}|\Gamma| \leq C(\int_{Q_1(x_0,t_0)} |\Gamma|^p \mathrm{d}x\mathrm{d}t)^{\frac{1}{p}}, \quad \textmd{$1\leq p<+\infty$}.
\end{equation}

Step 2. Prove the uniform decay \eqref{unidecay}.

We will use the integrality of $\Gamma$ in $L^\infty_tL^p(R^3\times(-\infty,0])$ to prove that when $r=|x'_0|=\sqrt{|x_{01}|^2+|x_{02}|^2}\rightarrow+\infty$, one has
\begin{equation}\label{unidecay}
(\int_{Q_1(x_0,t_0)}|\Gamma|^p \mathrm{d}x\mathrm{d}t)^{\frac{1}{p}}\leq Cr^{-\frac{1}{p}},
\end{equation}
uniformly for $t$ and $x_3$ (i.e. $z$).
It is obvious that
\[
(\int_{Q_1(x_0,t_0)}|\Gamma|^p \mathrm{d}x\mathrm{d}t)^{\frac{1}{p}}\leq \sup_{t\in (-\infty,0)}(\int_{B_1(x_0)}|\Gamma(\cdot,t)|^p \mathrm{d}x)^{\frac{1}{p}}.
\]
Thus it suffices to deal with the term $\int_{B_1(x_0)}|\Gamma(\cdot,t)|^p \mathrm{d}x$. Define a family of balls as $B_1(x)$ of radius 1, centered at $x$ where
$$x\in S=\{x=(x_1,x_2,x_3)|~|x'|=|x'_0|,x_3=x_{03}\}.$$
When $|x'_0|$ is sufficiently large, we can place at least $[|x'_0|]$ ($[|x'_0|]$ is the max integer less than $|x'_0|$) disjoint balls $B_1(x_i)$ on the curve $S$, where $1\leq i\leq [|x'_0|]$ and $x_i\in S$. Due to the axis-symmetry of $\Gamma$, we have
\[
\int_{B_1(x_i)}|\Gamma(\cdot,t)|^p \mathrm{d}x=\int_{B_1(x_0)}|\Gamma(\cdot,t)|^p \mathrm{d}x.
\]
Thus,
\begin{eqnarray}
[|x'_0|]\int_{B_1(x_0)}|\Gamma(\cdot,t)|^p \mathrm{d}x&=&\sum\limits^{[|x'_0|]}\limits_{i=0}\int_{B_1(x_i)}|\Gamma(\cdot,t)|^p \mathrm{d}x \nonumber \\
&\leq& \int_{\mathbb{R}^3}|\Gamma(\cdot,t)|^p \mathrm{d}x.\nonumber
\end{eqnarray}
It then follows that
\begin{equation}\label{decay}
 \begin{split}
(\int_{Q_1(x_0,t_0)}|\Gamma|^p \mathrm{d}x\mathrm{d}t)^{\frac{1}{p}}
\leq & ([|x'_0|])^{-\frac{1}{p}}\sup_{t\in (-\infty,0)}(\int_{\mathbb{R}^3}|\Gamma(\cdot,t)|^p\mathrm{d}x)^{\frac{1}{p}}\\
\leq & C([|x'_0|])^{-\frac{1}{p}} \|\Gamma\|_{L^\infty_tL^p_x}.
 \end{split}
\end{equation}
Since $\Gamma\in L^\infty_tL^p_x(\mathbb{R}^3\times(-\infty,0))$, we can deduce
from \eqref{decay} that \eqref{unidecay} holds.

After the above two steps, combining \eqref{meanvalue} and \eqref{unidecay} gives
\[
\sup\limits_{Q_{\frac{1}{2}}(x_0,t_0)}|\Gamma|\leq Cr^{-\frac{1}{p}}.
\]
Let $r\rightarrow \infty$, we then complete the proof of the lemma.

\end{proof}

Now let us prove Lemma \ref{lem5.2}.

\begin{proof}[Proof of Lemma \ref{lem5.2}]

Since $\Gamma=rv_\theta$, we have $\Gamma|_{r=0}=0$. We are going to use a contradiction argument and maximum principle to prove $\Gamma\equiv0$.
Suppose
\[
M=\sup_{(x,t)\in \mathbb{R}^3\times (-\infty,0)}\Gamma\geq0\geq\inf_{(x,t)\in \mathbb{R}^3\times (-\infty,0)}\Gamma=m.
\]
We will show $M=m=0$ in two cases.

Case 1. If $M>0$ and there exists a point $(x_0,t_0)$ $(|x_0|< \infty,|t|<\infty)$ such that $\Gamma(x_0,t_0)=M$.

First we notice that $r_0=\sqrt{(x_{01})^2+(x_{02})^2}\neq0$ because $\Gamma|_{r=0}=0$. Let $\widehat{\Gamma}=M-\Gamma$, then $\widehat{\Gamma}$ satisfies
\[
\partial_t\widehat{\Gamma}+b\cdot\nabla\widehat{\Gamma}
+\frac{2}{r}\partial_r\widehat{\Gamma}-\Delta\widehat{\Gamma}=0.
\]
Since $\Gamma$ satisfies the maximum principle in $(\mathbb{R}^3\setminus\{r=0\})\times(-\infty,0)$, we can see that $\Gamma\equiv M$ on $(\mathbb{R}^3\setminus\{r=0\})\times(-\infty,0)$, which induces a contradiction with $\Gamma\equiv0$ on $\{r=0\}\times(-\infty,0)$
and the fact that $\Gamma$ is continuous in $\mathbb{R}^3\times(-\infty,0)$.

Case 2. If $M>0$ and $\Gamma$ reaches $M$ when $|x|\rightarrow\infty$ or $t\rightarrow -\infty$.

In this case, there exists a sequence of points $x_k=(x'_k ,x_{k3})$ and a time sequence $t_k$ such that $\lim\limits_{k\rightarrow +\infty}\Gamma(x_k,t_k)=M$, where $|x_k|\rightarrow\infty$ or $t_k\rightarrow -\infty$.
Meanwhile, we know from Lemma \ref{lem5.1} that
$\lim\limits_{r\rightarrow +\infty}|\Gamma|=0$ holds uniformly for $t$ and $x_3$.
Thus we can assume that $|x'_k|\leq C$ where $C$ is independent of $t$ and $x_3$.
Without loss of generality, up to a subsequence, we can assume that $\lim\limits_{k\rightarrow +\infty}x'_k=(x_{01},x_{02})$. Because $\lim\limits_{k\rightarrow +\infty}\Gamma(x_k,t_k)=M>0$ and $\Gamma|_{r=0}=0$, we get $r_0=\sqrt{(x_{01})^2+(x_{02})^2}\neq 0$.

Let $v_k(y,s)=v(x_k+y,t_k+s)$, where $v$ is a bounded ancient mild solution of the axially symmetric Navier-Stokes equations, then $v_k(y,s)=v(x_k+y,t_k+s)$ is a bounded ancient mild solution of the following equations:
\begin{equation*}
 \begin{cases}
   \partial_sv_k+v_k\cdot\nabla v_k+\nabla p=\Delta v_k,\\
   \nabla\cdot v_k=0.
 \end{cases}
\end{equation*}
Let $v_\infty=\lim\limits_{k\rightarrow+\infty}v_k(y,s)$, then due to Lemma 6.1 in \cite{KNSS}, we know that $v_\infty$ is a bounded ancient mild solution of the limit equations. Since $v_k$ is axially symmetric with respect to the axis $\{x(x_1,x_2,x_3)| x_1=-x_{k1},x_2=-x_{k2},x_3\in \mathbb{R}\}$ and $\lim\limits_{k\rightarrow +\infty}x'_k=(x_{01},x_{02})$, we have $v_\infty$ is also axially symmetric and the axis is $\{x=(x_1,x_2,x_3)|x_1=-x_{01},x_2=-x_{02},x_3\in \mathbb{R}\}$.

Set $\widetilde{v}_{\infty}(y,s)=v_\infty(y-x_0,s)$, where $x_0=(x_{01},x_{02},0)$, then $\widetilde{v}_{\infty}(y,s)$ is symmetric with respect to the axis $\{x=(x_1,x_2,x_3)| x_1=x_2=0,x_3\in \mathbb{R}\}$.
Let $\widetilde{\Gamma}_\infty=r(y')\widetilde{v}_{\infty,\theta}$, where $r(y')=\sqrt{(y_1)^2+(y_2)^2}$, then $\widetilde{\Gamma}_\infty$ satisfies
\[
\partial_s\widetilde{\Gamma}_\infty+\widetilde{b}_\infty\cdot\nabla\widetilde{\Gamma}_\infty
+\frac{2}{r}\partial_r\widetilde{\Gamma}_\infty-\Delta\widetilde{\Gamma}_\infty=0,
\]
where $\widetilde{b}_\infty
=\widetilde{v}_{\infty,r}e_r+\widetilde{v}_{\infty,z}e_z$.

Now let $y-x_0=x$, $\widetilde{r}=\sqrt{(x_1+x_{01})^2+(x_2+x_{02})^2}$ and $\widetilde{\Gamma}=\widetilde{r}v_{\infty,\theta}$, one has
\begin{equation}\label{Gammatilde}
\partial_s\widetilde{\Gamma}+\widetilde{b}\cdot\nabla\widetilde{\Gamma}
+\frac{2}{\widetilde{r}}\partial_{\widetilde{r}}\widetilde{\Gamma}
-\Delta\widetilde{\Gamma}=0,
\end{equation}
where $\widetilde{b} =v_{\infty,r}e_r+v_{\infty,z}e_z$.
We know from the above definition that $\widetilde{\Gamma}(x_0,0)=0$ and
\begin{eqnarray}
\widetilde{\Gamma}(0,0)&=&\sqrt{(x_{01})^2+(x_{02})^2}v_{\infty,\theta}(0,0)\nonumber \\
                       &=&\sqrt{(x_{01})^2+(x_{02})^2}\lim\limits_{k\rightarrow +\infty}v_\theta(x_k,t_k)\nonumber \\
                       &=&\lim\limits_{k\rightarrow +\infty}\Gamma(x_k,t_k)=M.\nonumber
\end{eqnarray}
Since $\widetilde{\Gamma}$ satisfies the equation \eqref{Gammatilde}, one has that $\widetilde{\Gamma}$ satisfies the strong maximum principle in the whole space except $\{x=(x_1,x_2,x_3)| x_1=-x_{01},x_2=-x_{02},x_3\in \mathbb{R}\}$. Similar to Case 1, we get the contradiction.

Combining Case 1 and Case 2, we get $M=0$. Similarly, we can also get $m=0$. Hence, $\Gamma\equiv0$ in $(-\infty,0)\times \mathbb{R}^3$ which implies $v_\theta\equiv0$. According to Theorem 5.2 in \cite{KNSS}, we have that $v$ must be a constant and we complete the proof of the lemma. Meanwhile, Theorem \ref{3Dswirl} is obtained by Lemma \ref{lem5.1} and Lemma \ref{lem5.2}.

\end{proof}

\section*{Acknowledgement}
Z. Lei and N. Zhao was in part supported by NSFC (grant No. 11421061 and 11222107),
National Support Program for Young Top-Notch Talents, Shanghai Shu Guang project, and SGST 09DZ2272900. Qi S. Zhang was in part supported by the Simons Foundation and would like to thank School of Mathematics, Fudan University for supporting his visit. Particular thanks go to Prof. Hongjie Dong for useful discussion.


\end{document}